\documentclass[12pt]{article}
\usepackage{amsmath,amssymb,amsfonts}
\usepackage[top=1in,bottom=1in,right=1in,left=1in]{geometry}
\usepackage{setspace}
\usepackage{caption}
\usepackage{vwcol}
\usepackage{mathtools}
\usepackage{multicol}
\usepackage{amsthm}
\newtheorem{theorem}{Theorem}
\newtheorem{lemma}{Lemma}

\begin{document}
 
%%%------------------------------------------------------------------------------------------
%%%                                               ABSTRACT
%%%------------------------------------------------------------------------------------------

\begin{quote}
\begin{center} 
\textbf{The Discrete Rado Number for} \(x_1 + x_2 + \dots + x_m + c = 2x_0\)
\bigskip

Tristin Lehmann \\ Tristin.Lehmann@sdstate.edu \\
and\\

Donald L. Vestal, Jr \\ Donald.Vestal@sdstate.edu
\bigskip

South Dakota State University 
\bigskip

\textbf{Abstract}
\end{center}

\par\noindent
For a positive integer \(m\) and a real number \(c\), let \(R = R(m,c,2)\) denote the discrete 2-color Rado number for the equation  \(x_1 + x_2 + \dots + x_m + c = 2x_0\).  
In other words, \(R\) is the smallest integer such that for any coloring of the integers \({1, 2, \dots, R}\), there exist numbers \(x_1, x_2, \dots, x_m, x_0\), all with the same color, such that \(x_1 + x_2 + \dots + x_m + c = 2x_0\).
In this article we show that if \(m \geq 2\) and \(c > 0\), then 
\[
  R(m, c, 2) = 
    \begin{cases}
	\infty 													&	\text{if \(m\) is even and \(c\) is odd} \\
	\big\lceil \frac{m}{2} \big\lceil \frac{m+c}{2} \big\rceil + \frac{c}{2} \big\rceil		&	\text{otherwise~.}
    \end{cases}
 \]

\par\noindent
For real numbers \(a\) and \(c\), we look at the 2-color Rado number for the equation \(x_1 + c = ax_0\). 
We show that if \(a > 1\) and \(c > 0\), then the 2-color continuous Rado number is
\[
  R_\mathbb{R}(1, c, a) = 
    \begin{cases}
	\infty 					&	\text{if  \( a=1 \) } \\
	\frac{c}{a-1}	  			&	\text{otherwise.}
    \end{cases}
 \]
 From this, we will show that the discrete Rado number is
 \[
   R(1, c, a) = 
    \begin{cases}
    	\frac{c}{a-1}	  			&	\text{if \( \left( a-1 \right)  \mid c\)} \\ 
	\infty 					&	\text{otherwise.}
    \end{cases}
 \]

\end{quote}

%%%------------------------------------------------------------------------------------------
%%%                                               NARRATIVE
%%%------------------------------------------------------------------------------------------

%%%%%%%%%%%
\section{Introduction} %
%%%%%%%%%%%

In 1916, Isaai Schur proved in \cite{Schur} that for any coloring $\Delta : \mathbb{N} \rightarrow \{ 0 , 1, \dots, t-1 \} $ of the positive integers using \(t \geq 1\) colors, there exist positive integers \(x_1, x_2, x_3\) with $\Delta (x_1 ) = \Delta (x_2 ) = \Delta (x_3)$ such that \(x_1 + x_2 = x_3\).  
Such a solution is called a monochromatic solution. 
Consequently, for a given number of colors \(t\), there exists a smallest integer \(S = S(t)\) such that for any $t$-coloring of $\{ 1, 2, \dots, S \}$, there must be a monochromatic solution to the equation  \(x_1 + x_2 = x_3\).   
It is known that \(S(1)=2, S(2)=5, S(3)=14, S(4)=45\).  
Beyond \(t \geq 5\), the Schur numbers are unknown. 

Richard Rado in \cite{Rado1, Rado2, Rado3} generalized Schur's results to systems of linear equations.  
For a system $L$ of equations, the $t$-color Rado number for the system is the smallest integer $R$ such that for any $t$-coloring of the integers $\{ 1, 2, \dots , R \}$, there exists a monochromatic solution to the system $L$. 
If no such integer $R$ exists, then we say the $t$-color Rado number for $L$ is infinite. 

These contributions of Schur and Rado to Ramsey theory were existential; however, in 1982 in \cite{BB}, Beutelspacher and Brestovansky found specific values for a family of 2-color Rado numbers. In particular, the 2-color Rado number for the equation \(x_1 + x_2 + \dots + x_{m-1} = x_m\), where \( (m \geq 3) \), is \(m^2-m-1\).  
In other words, if you color the positive integers using two colors, say red and blue, then you can color from 1 to \(m^2 - m - 2\) and avoid a monochromatic solution to the given equation; but no matter how you color the integers from 1 to \(m^2 - m -1\) you are guaranteed to end up with a monochromatic solution.
Note that in the case of \(m=3\) this is \(S(2): m^2 - m -1 = 5\). 

After this result, there were many variations involving different equations, pairs of equations, inequalities, nonlinear equations, and some results involving more than two colors.

Burr and Loo also investigated a variation of Shur's equation in \cite{BL} by adding a constant \(c\) to obtain the equation \(x_1 + x_2 + c = x_3\).  
They proved that for \(c \in \mathbb{Z} \) the 2-color Rado number is \(4c+5\) if \(c \geq 0\), and \(|c| - \big\lceil \frac{|c|-5}{5} \big\rceil\) if \(c<0\).

These two works were generalized by Schaal when he investigated the equation \(x_1 + x_2 + \dots + x_{m-1} + c = x_m\).  In \cite{Sch} he found the 2-color Rado number to be: 
\[
  R = 
    \begin{cases}
	\infty 					&	\text{for \(m\) even and \(c\) odd} \\
	m^2 - m - 1 + (m+1)c 		&	\text{for \(c\geq 0 \), and \(m\) is odd or \(c\) is even}
    \end{cases}
 \]

In \cite{SchVes}, Schaal and Vestal found the 2-color Rado number for the equation \( x_1 + x_2 + \dots + x_{m-1} = 2x_m \) to be: 
\[
  R = 
    \begin{cases}
	\big\lceil \frac{m-1}{2} \big\lceil \frac{m-1}{2} \big\rceil \big\rceil	& \text{if }m \geq 6 \\
	5		& \text{if } m = 5 \\
	4		& \text{if } m = 4\\
	1		& \text{if } m = 3 \\
	\infty	& \text{if } m = 2 \\
    \end{cases}
 \]
 
Similar to this equation Vestal found the 2-color continuous Rado number for the equation \(x_1 + x_2 + \dots + x_m = ax_0 \) with \(a \in \mathbb{Z}\), \(a \geq 2\), and \(m \geq a(a-1) \), to be \(\frac{m^2}{a^2} \). \cite{Ves}
Continuous Rado numbers invovle coloring all of the positive real numbers greater than or equal to 1, as opposed to the discrete case where only the positive integers are colored.  

In this article, we consider the equation 
\[
\mathcal{E} : x_1+ x_2 + \dots + x_m + c = 2x_0.
\]

Let \(R = R(m,c,2)\) denote the 2-color Rado number for $\mathcal{E}$; that is, the smallest integer \(R\) such that for any 2-coloring of the integers in \({1, 2, \dots, R}\), there exists a monochromatic solution to the equation \(\mathcal{E}\).
We will prove the following:

\begin{theorem}
Let $m$ be a positive integer with $m\geq 2$ and let $c$ be a real number with \(c \geq 0\). Then
$$ \begin{array}{cc} 
R(m,c,2) =
 \begin{cases}
	\infty 													&	\text{for \(m\) even, \(c\) odd} \\
	\big\lceil \frac{m}{2} \big\lceil \frac{m+c}{2} \big\rceil + \frac{c}{2} \big\rceil		&	\text{otherwise.}
\end{cases}    \end{array} $$
\end{theorem}

%%%%%%%%%%%%
\section{Lower Bound} %
%%%%%%%%%%%%

We will show the stronger case of \( R = \big\lceil \frac{m}{a} \big\lceil \frac{m+c}{a} \big\rceil + \frac{c}{a} \big\rceil \) being a lower bound of \(R ( m, c, a ) \) for the equation \(\mathcal{E}\).  
To prove the lower bound we must show that we can color the integers in \([1,R-1]\) using two colors, say red and blue, and avoid a monochromatic solution (i.e. not every \(x_i\) can be the same color).
For any \(n \in \mathbb{N}\), let \([1,n]\) denote the set \(\{1, 2, \dots, n\}\).

\begin{lemma}
For \(m \geq 2\), \(c > 0\), and \(a \geq 1\) we have \(R(m,c,a) \geq \big\lceil \frac{m}{a} \big\lceil \frac{m+c}{a} \big\rceil + \frac{c}{a} \big\rceil \).
\end{lemma}

\begin{proof}  Consider the following coloring: color \( [1, \lceil \frac{m+c}{a}\rceil  - 1 ] \) red and \( \Big[\lceil \frac{m+c}{a}\rceil, \big\lceil \frac{m}{a} \big\lceil \frac{m+c}{a} \big\rceil + \frac{c}{a}  \big\rceil - 1  \Big] \) blue.  We claim that this coloring avoids a monochromatic solution.

Suppose that \(x_1, \dots, x_m\) are red.  Then for \(i=1, \dots, m\), we have \(x_i \geq 1\), and so

\[
\begin{aligned}
x_0 &= \frac{1}{a}(x_1 + x_2 + \dots + x_m + c) \\
	& \geq \frac{1}{a}(1 + 1 + \dots + 1 + c) \\
	& = \frac{1}{a}(m + c) \\
	& > \Big\lceil \frac{m + c}{a} \Big\rceil - 1\,.
\end{aligned}
\]
Therefore \(x_0\) is not red and a red solution is avoided.

Now suppose that \(x_1, \dots, x_m\) are blue.  
Then for \(i=1, \dots, m\), we have \(x_i \geq \big\lceil \frac{m+c}{a} \big\rceil \), so

\[
\begin{aligned}
x_0 &= \frac{1}{a}(x_1 + x_2 + \dots + x_m + c) \\
	& \geq \frac{1}{a}\Big(\Big\lceil \frac{m+c}{a} \Big\rceil + \Big\lceil \frac{m+c}{a} \Big\rceil + \dots + \Big\lceil \frac{m+c}{a} \Big\rceil + c\Big) \\
	& = \frac{1}{a}\Big(m\Big\lceil \frac{m+c}{a} \Big\rceil+ c\Big) \\
	& >  \Big\lceil \frac{m}{a} \Big\lceil \frac{m+c}{a} \Big\rceil + \frac{c}{a} \Big\rceil - 1\,.
\end{aligned}
\]
Therefore \(x_0 \geq  R\) and a monochromatic solution is avoided. 

Thus \( \big\lceil \frac{m}{a} \big\lceil \frac{m+c}{a} \big\rceil + \frac{c}{a} \big\rceil \) is a lower bound for \(R(m,c,a)\).
\end{proof}

%%%%%%%%%%%%
\section{Upper Bound}%
%%%%%%%%%%%%

We will now show that \( R = \big\lceil \frac{m}{2} \big\lceil \frac{m+c}{2} \big\rceil + \frac{c}{2} \big\rceil \) is also an upper bound for $ R ( m, c, 2 )$ and therefore is the Rado number.  
This will be done by showing that for any 2-coloring of the integers in \([1,R]\), there must be a monochromatic solution to our equation.
In the proof, we will be looking at solutions to the equation $ \mathcal{E}$, so we will adopt the following notation: solutions will be given as $(m+1)$-tuples $ \left( x_1 , x_2 , \dots, x_m , x_0 \right) $. When multiple variables are assigned the same value, we will use underbraces to denote the number of such variables. For example, the $(m+1)$-tuple 
$$ ( \underbrace{ 1, \dots, 1 }_\text{\(m-5\)}, \underbrace{ 2, \dots, 2 }_\text{\(4\)}, m+c+3,m+c+3 ) $$
denotes the solution with $x_1 = x_2 = \cdots = x_{m-5} = 1, x_{m-4} =x_{m-3} = \cdots = x_{m-1} = 2$, and $x_m = x_0 = m + c +3$.

\begin{lemma}
For \(m \geq 2\) and \(c> 0\), we have 
\[
  R(m, c, 2) \leq
    \begin{cases}
	\infty 													&	\text{if \(m\) is even and \(c\) is odd} \\
	\big\lceil \frac{m}{2} \big\lceil \frac{m+c}{2} \big\rceil + \frac{c}{2} \big\rceil		&	\text{otherwise~.}
    \end{cases}
 \]
\end{lemma}

To show this lemma, we will break our problem into several cases: 
\begin{multicols}{2}
\begin{enumerate}
\item[] Case I: \(m\) even and \(c\) odd 
\item[] Case II: \(m\) even and \(c\) even 
	\begin{enumerate}
	\item[] Case II.A: 1 and 2 Red
	\item[] Case II.B: 1 Red and 2 Blue
		\begin{enumerate}
			\item[] Case II.B.1: 3 Red
			\item[] Case II.B.2: 3 Blue
		\end{enumerate}
	\end{enumerate}
\item[] Case III: \(m\) odd and \(c\) odd 
	\begin{enumerate}
	\item[] Case III.A: 1 and 2 Red
		\begin{enumerate}
		\item[] Case III.A.1: \( s < \frac{R - c}{m-1} \) and \(ms+c-m+2\) even
		\item[] Case III.A.2: \( s < \frac{R - c}{m-1} \) and \(ms+c-m+2\) odd
		\item[] Case III.A.3: \( s \geq \frac{R - c}{m-1} \)
		\end{enumerate}
	\item[] Case III.B: 1 Red and 2 Blue
		\begin{enumerate}
		\item[] Case III.B.1: 3 Red
		\item[]  Case III.B.2: 3 Blue
		\end{enumerate}
	\end{enumerate}
\item[] Case IV: \(m\) odd and \(c\) even
	\begin{enumerate}
	\item[] Case IV.A: 1 and 2 Red
		\begin{enumerate}
		\item[] Case IV.A.1: \( s <  \frac{R - c}{m-1} \)
		\item[]  Case IV.A.2: \( s \geq  \frac{R - c}{m-1} \)
		\end{enumerate}
	\item[] Case IV.B: 1 Red and 2 Blue
		\begin{enumerate}
		\item[] Case IV.B.1: 3 Red
		\item[] Case IV.B.2: 3 Blue
		\end{enumerate}
	\end{enumerate}
\end{enumerate}
\end{multicols}

%%%%%%%%%%%%%%%%%%
\subsubsection*{\large{Case I: $m$ even, $c$ odd}}%
%%%%%%%%%%%%%%%%%%

 When \(m\) is even and \(c\) is odd, we claim the Rado number is infinite.

\begin{proof}  Consider the following coloring: color the even integers red and the odd integers blue.
Let \(x_1, \dots, x_m\) be red and therefore even.  
Then \(x_1 + x_2 + \dots + x_m\) is even, and so \(x_1 + x_2 + \dots + x_m + c\) is odd.
Then \(x_1 + x_2 + \dots + x_m + c = 2x_0\) has no integer solution for \(x_0\).
Similarly let \(x_1, \dots, x_m\) be blue and therefore odd.  
Then  \(x_1 + x_2 + \dots + x_m\) is even, and so \(x_1 + x_2 + \dots + x_m + c\) is odd.
Again there is no integer solution for \(x_0\).
Therefore we can color all of the natural numbers and avoid a monochromatic solution.
(This is a similar argument to that of Schaal and Vestal's on the discrete case for \(x_1 + x_2 + \dots + x_{m-1} = 2x_m\).  \cite{SchVes})
\end{proof}

\newpage
%%%%%%%%%%%%%%%%%%%
\subsubsection*{\large{Case II: $m$ even, $c$ even}}%
%%%%%%%%%%%%%%%%%%%

Note that when both \(m\) and \(c\) are even, we get \( R = \frac{m^2 + mc}{4} + \frac{c}{2} \).  

\subsubsection*{Case II.A: 1 and 2 Red}

\begin{proof}  Suppose that \( 1, 2, \dots, s \) are colored red and \( s+1\) is colored blue; thus \(s+1\) is the first blue number.  
If such an \(s\) does not exist, then all of the integers in \([1,R]\) are colored red and we would have a red solution, namely \( (1, \dots, 1, \frac{m+c}{2} ) \).
Thus we see that \(\frac{m+c}{2}\) is blue and thus \( s+1 \leq \frac{m+c}{2} \). This implies that \( s \leq \frac{m+c}{2} -1 \), which we will use below.
Now consider \( x = \frac{s(m-2)}{2}+m+c-1 \).  

Note that  \( (1, \underbrace{ 2, \dots, 2 }_\text{ \( \frac{m-2}{2} \) }, \underbrace{ s, \dots, s }_\text{ \( \frac{m-2}{2} \) }, x, x ) \) is a solution which implies that \(x\) is blue and \( (s+1, \dots, s+1, \frac{m+c}{2}, \frac{m+c}{2}, x ) \) is a solution implying that \(x\) is red.  
This contradiction is enough to show that \( \frac{m^2 + mc}{4} + \frac{c}{2} \) is our Rado number, provided that \(x \leq R\).  This is true since

\[
\begin{aligned}
	x &= \Big(\frac{m-2}{2}\Big)s+m+c-1 \\
	   &\leq \bigg(\frac{m-2}{2}\bigg) \bigg( \frac{m+c}{2}-1 \bigg)+m+c-1 \\
	   &= \frac{m^2 + mc}{4} + \frac{c}{2}~.
\end{aligned}
\]
Note that when \(s = \frac{m+c}{2} -1\) we get \(x = R\).
\end{proof}

\subsubsection*{Case II.B: 1 Red and 2 Blue} 

To cover this case we look at the two possible colorings for 3, red and blue, which gives us two sub cases.

\subparagraph{Case II.B.1: 3 Red}
\begin{proof} 

Consider \( x = \frac{2m+c}{2} \).  Note that \( x \leq \frac{m^2 + mc}{4} + \frac{c}{2} = R \).

The solution \( ( \underbrace{ 1, \dots, 1 }_\text{ \( \frac{m}{2} \) }, \underbrace{ 3, \dots, 3 }_\text{ \( \frac{m}{2} \) }, x ) \) implies that \(x\) is blue, and the solution \( (2, \dots, 2, x) \) implies that \(x\) is red, thus giving us a contradiction.

\end{proof}
\subparagraph{Case II.B.2: 3 Blue} \label{II.B.2} 
\begin{proof}
The solution \( (2, \dots, 2, c+2m-2, c+2m-2 ) \) implies that \(c+2m-2\) is red, and \(c+3m-3\) is red since \( (3, \dots, 3, c+3m-3, c+3m-3 )\) is a solution.  
Then \( (1, \dots, 1, c+3m-3, c+2m-2 ) \) creates a red solution.

We must also show that \(c+2m-2\) and \(c+3m-3\) are less than \( R  \).  
Since \(c+2m-2 \leq c+3m-3\) for \(m \geq 1\), we only need to verify the inequality involving \(c+3m-3\).

\newpage

Note that:
\begin{table}[h!]
\begin{center}
\begin{tabular}{l r c l r} 
		& \( 3m + c - 3 \)  	& \( \leq \) 	& \( \displaystyle\frac{m^2 + mc}{4} + \frac{c}{2} \) 	& \\
\( \iff \)   	& 0				& \( \leq \)	& \( m^2 + mc - 12m - 2c + 12 \) 					& \\
\( \iff	\)	& 0				& \( \leq \)	& \( c(m - 2) + (m^2 - 12m + 12) ~. 	\)				& (1) 
\end{tabular}
\end{center}
\end{table}

Since \(c(m-2) \geq 0\) it suffices to have \((m^2 - 12m + 12) \geq 0\), but this is equivalent to:
\begin{align*}
	24 	  &\leq (m-6)^2 \\
	11	  &\leq m~.
\end{align*}

Thus \( m < 11 \) is not covered by the above proof and we must prove each case separately.  Using (1) we can determine for which \(c\) values our above argument holds and which we have to prove.  We use this to show there is no way to avoid a monochromatic solution and therefore \(R\) is the Rado number.  

For each case assume 1 is red, and 2 and 3 are blue.

\subparagraph{\underline{$m=10$:}} \indent \( 0 \leq 8c - 8 \) implies that \( c \geq 1 \), which is always true; thus no separate case is needed.

\subparagraph{\underline{$m=8$:}} \indent \( 0 \leq 6c - 20 \) implies that \( c \geq 4 \).  Thus we need only consider \(c=2\).

\begin{table}[!h]
\begin{tabular}{lll}
\hspace{2.7cm} &\(R(8,2,2)=21\) & \( (2, 2, 2, 2, 2, 2, 2, 2, 9) \Rightarrow \) 9 is red                \\
&            & \( (3, 3, 3, 3, 3, 3, 3, 3, 13) \Rightarrow \) 13 is red              \\
&            & \( (1, 1, 1, 1, 1, 1, 9, 9, 13) \) creates a red solution.
\end{tabular}
\end{table}

\subparagraph{\underline{$m=6$:}} \indent \( 0 \leq 4c - 24 \) implies that \( c \geq 6 \).  Thus we consider \(c=2\) and \(c=4\).

\begin{table}[!h]
\begin{tabular}{lll}
\hspace{2.7cm} & \(R(6,2,2)=13\) & \( (2, 2, 2, 2, 2, 2, 7) \Rightarrow \) 7 is red           \\
                 &                 & \( (3, 3, 3, 3, 3, 3, 10) \Rightarrow \) 10 is red         \\
                 &                 & \( (1, 1, 1, 1, 7, 7, 10) \) creates a red solution.  \\
                 & \(R(6,4,2)=17\) & \( (1, 1, 1, 1, 1, 1, 5) \Rightarrow \) 5 is blue          \\
                 &                 & \( (2, 2, 2, 2, 2, 2, 8) \Rightarrow \) 8 is red           \\
                 &                 & \( (3, 3, 3, 3, 3, 3, 11) \Rightarrow \) 11 is red         \\
                 &                 & \( (5, 5, 5, 5, 5, 5, 17) \Rightarrow \) 17 is red         \\
                 &                 & \( (1, 1, 1, 8, 8, 11, 17) \) creates a red solution.
\end{tabular}
\end{table}

\newpage
\subparagraph{\underline{$m=4$:}} \indent \( 0 \leq 2c - 20 \) implies that \( c \geq 10 \).  Thus we consider \(c=2, 4, 6, \text{and } 8\).

\begin{table}[h!]
\begin{tabular}{lll}
\hspace{2.7cm} & \(R(4,2,2)=7\)  & \( (2, 2, 2, 2, 5) \Rightarrow \) 5 is red      \\
                 &                 & \( (3, 3, 3, 3, 7) \Rightarrow \) 7 is red      \\
                 &                 & \( (1, 1, 5, 5, 7) \) creates a red solution.   \\    
	         & \(R(4,4,2)=10\) & \( (1, 1, 1, 1, 4) \Rightarrow \) 4 is blue     \\
                 &                 & \( (2, 2, 2, 2, 6) \Rightarrow \) 6 is red      \\
                 &                 & \( (3, 3, 3, 3, 8) \Rightarrow \) 8 is red      \\
                 &                 & \( (4, 4, 4, 4, 10) \Rightarrow \) 10 is red      \\
                 &                 & \( (1, 1, 6, 8, 10) \) creates a red solution.       \\   
               & \(R(4,6,2)=13\) & \( (1, 1, 1, 1, 5) \Rightarrow \) 5 is blue     \\
                 &                 & \( (2, 2, 2, 2, 7) \Rightarrow \) 7 is red      \\
                 &                 & \( (3, 3, 3, 3, 9) \Rightarrow \) 9 is red      \\
                 &                 & \( (5, 5, 5, 5, 13) \Rightarrow \) 13 is red    \\
                 &                 & \( (1, 1, 9, 9, 13) \) creates a red solution.  \\
                 & \(R(4,8,2)=16\) & \( (1, 1, 1, 1, 6) \Rightarrow \) 6 is blue     \\
                 &                 & \( (2, 2, 2, 2, 8) \Rightarrow \) 8 is red      \\
                 &                 & \( (1, 1, 9, 9, 13) \Rightarrow \) 13 is blue   \\
                 &                 & \( (3, 3, 6, 6, 13) \) creates a blue solution.
\end{tabular}
\end{table}

\subparagraph{\underline{$m=2$:}} \indent (1) fails in this case; therefore we need to consider the equation \(x_1 + x_2 + c = 2x_0\) for arbitrary even integers \(c\).

\begin{table}[h]
\begin{tabular}{lll}
\hspace{2.7cm} & \(R(2,c,2)=c+1\) & \((1, c+1, c+1) \Rightarrow c+1\) is blue                                \\
                 &                  & \( (1, 1, \frac{c+2}{2}) \Rightarrow \frac{c+2}{2}\) is blue       \\
                 &                  & \( (\frac{c+2}{2}, \frac{c+2}{2}, c+1) \) creates a blue solution.
\end{tabular}
\end{table}
\end{proof}

%%%%%%%%%%%%%%%%%%
\subsubsection*{\large{Case III: $m$ odd, $c$ odd}}%
%%%%%%%%%%%%%%%%%%

%
\subsubsection*{Case III.A: 1 and 2 Red} 

We will consider three sub cases based on the values of \( \frac{R - c}{m-1} \)  and \(ms+c-m+2\). 

\subparagraph{Case III.A.1: \( s < \frac{R - c}{m-1} \) and  \( ms+c-m+2\) even } \label{III.A.1}
\begin{proof} Suppose that 1 and 2 are red.  Let  \( 1, 2, \dots, s \) be colored red and \(s+1\) be colored blue.  
The solution \( ( s-1, \dots, s-1, s, s, \frac{ms-m+c+2}{2} ) \) implies that \( \frac{ms-m+c+2}{2} \) is blue.

The solution \( (s+1, \dots, s+1, \frac{ms+c-m+2}{2}, \frac{ms+c-m+2}{2}, c+(m-1)s ) \) implies that \(c+(m-1)s\) is red.  Note that \(c+(m-1)s < R\) since this is equivalent to \(s < \frac{R - c}{m-1} \).
Then \( (s, \dots, s, c+(m-1)s, c+(m-1)s ) \) creates a red solution and a contradiction.
\end{proof}
\subparagraph{Case III.A.2: \( s < \frac{R - c}{m-1} \) and  \( ms+c-m+2\) odd }
\begin{proof}
The solutions \( ( s-1, \dots, s-1, s, \frac{ms+c-m+1}{2} ) \) and \( ( s-1, \dots, s-1, s, s, s, \frac{ms+c-m+3}{2} ) \) imply that both \( \frac{ms+c-m+1}{2} \) and \(  \frac{ms+c-m+3}{2} \) are blue.  The solution
$$ ( s+1, \dots, s+1,  \frac{ms+c-m+1}{2}, \frac{ms+c-m+3}{2}, c+(m-1)s ) $$
implies that \( c+(m-1)s \) is red.
Then  \( (s, \dots, s, c+(m-1)s, c+(m-1)s ) \) creates a red solution and a contradiction.
\end{proof}
\subparagraph{Case III.A.3: \( s \geq  \frac{R - c}{m-1} \)} \label{III.A.2}
\begin{proof}
The solutions \( (1, \dots, 1, 1, \frac{m+c}{2}) \) and \( (1, \dots, 1, 2, 2, \frac{m+c}{2}+1 ) \) imply that \( \frac{m+c}{2} \) and \( \frac{m+c}{2}+1 \) are blue. 
With \(m\) odd and \(c\) odd we have two cases: \( m+c \equiv 0 \pmod{4} \) or \( m+c \equiv 2 \pmod{4} \). 

When \( m+c \equiv 0 \pmod{4} \) we have \(R = \big\lceil \frac{m}{2} \big\lceil \frac{m+c}{2} \big\rceil + \frac{c}{2} \big\rceil =  \big\lceil \frac{m}{2} \big( \frac{m+c}{2} \big) + \frac{c}{2} \big\rceil = \) \(
\frac{m^2 + cm}{4} + \frac{c+1}{2}\) and the solution \( ( \frac{m+c}{2}, \dots, \frac{m+c}{2}, \frac{m+c}{2}+1, R) \) implies that \( R\) is red.
Similarly, when \( m+c \equiv 2 \pmod{4} \) we have \(R = \frac{m^2 + cm}{4} + \frac{c}{2}\) and the solution \( ( \frac{m+c}{2}, \dots, \frac{m+c}{2}, R) \) implies that \(R\) is red.

Since \( s \geq \frac{R-c}{m-1} \) we know that \( s \geq \lceil\frac{R-c}{m-1}\rceil \) and \( s > \lfloor\frac{R-c}{m-1}\rfloor \), thus \( \lceil\frac{R-c}{m-1}\rceil \) and \(  \lfloor\frac{R-c}{m-1}\rfloor \) are both red.
To show that a red solution exists, consider \[ \Big( \underbrace{ \Big\lfloor\frac{R-c}{m-1}\Big\rfloor, \dots, \Big\lfloor\frac{R-c}{m-1}\Big\rfloor}_\text{\(x\)}, \underbrace{ \Big\lceil\frac{R-c}{m-1}\Big\rceil, \dots,\Big\lceil\frac{R-c}{m-1}\Big\rceil }_\text{\(y\)}, R, R \Big). \]
This will be a red solution if \( \lfloor\frac{R-c}{m-1}\rfloor x + \lceil\frac{R-c}{m-1}\rceil y = R-c \) and \( x+y = m-1\).

Note that if \(x=m-1\) and \(y=0\) then \(  \lfloor\frac{R-c}{m-1}\rfloor(m-1) \leq \frac{R-c}{m-1}(m-1) = R-c \).
Similarly, if \(x=0\) and \(y=m-1\) then \( \lceil\frac{R-c}{m-1}\rceil(m-1) \geq \frac{R-c}{m-1}(m-1) = R-c\ \).

Thus we can take \(x_1 = x_2 = \dots = x_{m-1} = \lfloor\frac{R-c}{m-1}\rfloor \) and replace a \( \lfloor\frac{R-c}{m-1}\rfloor \) with \( \lceil\frac{R-c}{m-1}\rceil \) to increase our value by 1.  The discrete intermediate value theorem guarantees that we can continue doing this and eventually reach the value \(R-c\).
In particular we find these values to be: 
\begin{align*}
x &= c+\Big\lceil\frac{R-c}{m-1}\Big\rceil(m-1) - R \\
y &=  m - 1 - c - \Big\lceil\frac{R-c}{m-1}\Big\rceil(m-1) + R ~.
\end{align*}
\end{proof}

\subsubsection*{Case III.B: 1 Red and 2 Blue}

Similar to Case II.B, on page \pageref{II.B.2}, we will look at two possible colorings for 3, which gives us two sub cases.

\subparagraph{Case III.B.1: 3 Red} \label{III.2.a}

\begin{proof}  Consider \( x = 2(m-1) + c \). 
The solution \( (\underbrace{ 1, \dots, 1 }_\text{\(\frac{m-1}{2}\)}, \underbrace{ 3, \dots, 3 }_\text{\(\frac{m-1}{2}\)}, x, x ) \) implies that \(x\) is blue, and \( (2, \dots, 2, x, x ) \) implies that \(x\) is red; thus creating a contradiction.
We must also show that \( x \leq R \).
From Case III.A.3, we know that the smallest \(R\) can be is \( \frac{m^2 + cm}{4}  + \frac{c}{2}\).
So it suffices to show \( 2(m-1)+c \leq \frac{m^2 +cm}{4} + \frac{c}{2} \).

\bigskip
Note that:
\begin{table}[h!]
\begin{center}
\begin{tabular}{l r c l r} 
		& \( 2(m-1)+c \) 	& \( \leq \) 	& \( \displaystyle\frac{m^2+cm}{4} + \frac{c}{2} \)	& \\
\( \iff	\)	& 0	   			& \( \leq \)	& \( m^2 + mc - 8m - 2c + 8 \)					& \\
\( \iff	\)	& 0	   			& \( \leq \)	& \( c(m - 2) + (m^2 - 8m + 8) ~. \)				& (2)
\end{tabular}
\end{center}
\end{table}

Since \(c(m-2) \geq 1\) it suffices to have \((m^2 - 8m + 8) \geq -1\), but this is equivalent to:
\begin{align*}
	7	  &\leq (m-4)^2 \\
	7	  &\leq m~.
\end{align*}

Thus \( m < 7 \) is not covered  above and we must prove each case separately.  Using (2) we can determine for which \(c\) values our above argument holds and which we have to prove.  

For each case assume 1 and 3 are red, and 2 is blue.

\subparagraph{\underline{$m=5$:}} \indent \( 0 \leq 3c - 7 \) implies that \( c \geq 2 \).  Thus we only consider \(c=1\).

\begin{table}[!h]
\begin{tabular}{lll}
\hspace{2.7cm} &\(R(5,1,2)=8\) & \( (1, 1, 1, 1, 1, 3) \) creates a red solution.
\end{tabular}
\end{table}

\subparagraph{\underline{$m=3$:}} \indent \( 0 \leq c - 7 \) implies that \( c \geq 7 \).  Thus we consider \(c=1, 3, \text{and } 5\).

\begin{table}[!h]
\begin{tabular}{lll}
\hspace{2.7cm} & \(R(3,1,2)=4\)  & \( (1, 1, 3, 3) \) creates a red solution.    \\
                 & \(R(3,3,2)=6\) & \( (1, 1, 1, 3) \) creates a red solution.    \\
                 & \(R(3,5,2)=9\)& \( (1, 1, 3, 5) \Rightarrow \) 5 is blue     \\
                 &		    & \( (3, 3, 3, 7 ) \Rightarrow \) 7 is blue \\
                 &                 & \( (2, 2, 5, 7) \) creates a blue solution.
\end{tabular}
\end{table}

\end{proof}
\subparagraph{Case III.B.2: 3 Blue }\label{III.2.b}
\begin{proof}
The solution \( (2, \dots, 2, c+2m-2, c+2m-2 ) \) implies that \(c+2m-2\) is red, and \(c+3m-3\) is red since \( (3, \dots, 3, c+3m-3, c+3m-3 )\) is a solution.  
Then \( (1, \dots, 1, c+3m-3, c+2m-2 ) \) creates a red solution and a contradiction.

We must also show that \(c+2m-2\) and \(c+3m-3\) are less than \( R = \frac{m^2 + mc}{4} + \frac{c}{2} \).  
Since \(c+2m-2 \leq c+3m-3\) for \(m \geq 1\), we only need to verify the inequality involving \(c+3m-3\).

\newpage
Note that:
\begin{table}[h!]
\begin{center}
\begin{tabular}{l r c l r} 
		& \( 3m + c - 3 \)	& \( \leq \)	& \( \displaystyle\frac{m^2 + mc}{4} + \frac{c}{2} \)	& \\
\( \iff	\) 	& 0	   			& \( \leq \)	& \( m^2 + mc - 12m - 2c + 12 \)		& \\
\( \iff	\) 	& 0	   			& \( \leq \)	& \( c(m - 2) + (m^2 - 12m + 12) ~. \)	& (3)
\end{tabular}
\end{center}
\end{table}

Since \(c(m-2) \geq 1\) it suffices to have \((m^2 - 12m + 12) \geq -1\), but this is equivalent to:
\begin{align*}
	23 	  &\leq (m-6)^2 \\
	11	  &\leq m~.
\end{align*}

Thus \( m < 11 \) is not covered by the above argument and so we must prove each case separately.  Using (3) we can determine for which \(c\) values our above argument holds and which we have to prove. 

For each case assume 1 is red, and 2 and 3 are blue.

\subparagraph{\underline{$m=9$:}} \indent \( 0 \leq 7c - 15 \) implies that \( c \geq 2 \).  Thus we need only consider \(c=1\)

\begin{table}[!h]
\begin{tabular}{lll}
\hspace{2.7cm} &\(R(9,1,2)=23\) & \( (2, 2, 2, 2, 2, 2, 2, 2, 3,10) \Rightarrow \) 10 is red                \\
&            & \( (3, 3, 3, 3, 3, 3, 3, 3, 3, 14) \Rightarrow \) 14 is red             \\
&            & \( (1, 1, 1, 1, 1, 1, 1, 10, 10, 14) \) creates a red solution.
\end{tabular}
\end{table}

\subparagraph{\underline{$m=7$:}} \indent \( 0 \leq 5c - 23 \) implies that \( c \geq 5 \).  Thus we need only consider \(c=1\) and \(c=3\).

\begin{table}[!h]
\begin{tabular}{lll}
\hspace{2.7cm} &\(R(7,1,2)=15\) & \( (2, 2, 2, 2, 2, 2, 3, 8) \Rightarrow \) 8 is red                \\
&            & \( (3, 3, 3, 3, 3, 3, 3, 11) \Rightarrow \) 11 is red              \\
&            & \( (1, 1, 1, 1, 1, 8, 8, 11) \) creates a red solution.		  \\
&	       \(R(7,3,2)=19\) & \( (2, 2, 2, 2, 3, 3, 3, 10 \Rightarrow \) 10 is red			\\
&            & \( (2, 2, 3, 3, 3, 3, 3, 11) \Rightarrow \) 11 is red              \\
&            & \( (1, 1, 1, 1, 1, 1, 11, 10) \) creates a red solution.
\end{tabular}
\end{table}

\subparagraph{\underline{$m=5$:}} \indent \( 0 \leq 3c - 23 \) implies that \( c \geq 8 \).  Thus we consider \(c=1, 3, 5, \) \( \text{and } 7\).

\begin{table}[!h]
\begin{tabular}{lll}
\hspace{2.7cm} & \(R(5,1,2)=8\) & \( (2, 2, 2, 2, 3, 6) \Rightarrow \) 6 is red           \\
                 &                 & \( (3, 3, 3, 3, 3, 8) \Rightarrow \) 8 is red         \\
                 &                 & \( (1, 1, 1, 6, 6, 8) \) creates a red solution.  \\
                 & \(R(5,3,2)=12\) & \( (2, 2, 2, 2, 3, 7) \Rightarrow \) 7 is red          \\
                 &                 & \( (1, 1, 1, 1, 7, 7) \) creates a red solution.  
\end{tabular}
\end{table}
\begin{table}[!h]
\begin{tabular}{lll}
\hspace{2.7cm}		& \(R(5,5,2)=15\) & \( (2, 2, 3, 3, 3, 9) \Rightarrow \) 9 is red          \\
                 &                 & \( (1, 1, 1, 1, 9, 9) \) creates a red solution.   \\
                 & \(R(5,7,2)=19\) & \( (3, 3, 3, 3, 3, 11) \Rightarrow \)11 is red          \\
                 &                 & \( (1, 1, 1, 1, 11, 11) \) creates a red solution.
\end{tabular}
\end{table}

\newpage
\subparagraph{\underline{$m=3$:}} \indent \( 0 \leq c - 15 \) implies that \( c \geq 15 \).  Thus we consider \(c=1, 3, 5, 7, \) \( 9, 11,\text{ and } 13\).

\begin{table}[!h]
\begin{tabular}{lll}
\hspace{2.7cm} & \(R(3,1,2)=5\)  & \( (2, 2, 3, 4) \Rightarrow \) 4 is red      \\
                 &                 & \( (1, 4, 4, 5) \Rightarrow \) 5 is blue      \\
                 &                 & \( (2, 2, 5, 5) \) creates a blue solution.   \\
                 & \(R(3,3,2)=6\) & \( (2, 2, 3, 5) \Rightarrow \) 5 is red     \\
                 &                 & \( (1, 1, 5, 5) \) creates a red solution.       \\
                 & \(R(3,5,2)=9\)& \( (3, 3, 3, 7) \Rightarrow \) 7 is blue     \\
                 &                 & \( (1, 1, 7, 7) \) creates a red solution.  \\
            & \(R(3,7,2)=11\) & \( (1, 1, 1, 5) \Rightarrow \) 5 is blue     \\
                 &                 & \( (2, 2, 3, 7) \Rightarrow \) 7 is red      \\
                 &                 & \( (5, 5, 5, 11),\Rightarrow \) 11 is red   \\
                 &                 & \( (1, 7, 7, 11) \) creates a red solution.   \\
                 & \(R(3,9,2)=14\) & \( (1, 1, 1, 6) \Rightarrow \) 6 is blue     \\
                 &                 & \( (3, 3, 3, 9) \Rightarrow \) 9 is red      \\
                 &                 & \( (1, 1, 9, 10) \Rightarrow \) 10 is blue    \\
                 &                 & \( (2, 3, 6, 10) \) creates a blue solution.  \\
                  & \(R(3,11,2)=16\) & \( (1, 1, 1, 7) \Rightarrow \) 7 is blue     \\
                 &                 & \( (2, 2, 3, 9) \Rightarrow \) 9 is red      \\
                 &                 & \( (1, 1, 9, 11) \Rightarrow \) 11 is blue    \\
                 &                 & \( (2, 2, 7, 11) \) creates a blue solution.  \\
                  & \(R(3,13,2)=19\) & \( (1, 1, 1, 8) \Rightarrow \) 8 is blue     \\
                 &                 & \( (3, 3, 3, 11) \Rightarrow \)11 is red      \\
                 &                 & \( (1, 1, 11, 13) \Rightarrow \) 13 is blue    \\
                 &                 & \( (2, 3, 8, 13) \) creates a blue solution.  \\
\end{tabular}
\end{table}

\end{proof}

%%%%%%%%%%%%%%%%%%
\subsubsection*{\large{Case IV: $m$ odd, $c$ even}}%
%%%%%%%%%%%%%%%%%%

%
\subsubsection*{Case IV.A: 1 and 2 Red}

Similar to case III.A we will consider the sub cases based on the value of \( \frac{R - c}{m-1} \), except that we can combine two of the previous three cases for this proof.

\subparagraph{Case IV.A.1: \( s < \frac{R - c}{m-1} \)}

\begin{proof} This proof is identical to Case III.A.1 and Case III.A.2 on page \pageref{III.A.1}. 

\subparagraph{Case IV.A.2: \( s \geq  \frac{R - c}{m-1} \)} 

This proof is similar to Case III.A.3, on page \pageref{III.A.2}.

The solutions \( (1, \dots, 1, 1, \frac{m+c+1}{2}) \) and \( (1, \dots, 1, 2, 2, \frac{m+c+3}{2} ) \) imply that \( \frac{m+c+1}{2} \) and \( \frac{m+c+3}{2} \) are blue. 
With \(m\) odd and \(c\) odd we have two cases: \( m+c \equiv 1 \pmod{4} \) or \( m+c \equiv 3 \pmod{4} \). 

When \( m+c \equiv 1 \pmod{4} \) then \(R = \frac{m^2 + cm+m}{4} + \frac{c+1}{2}\) and the solution \\
 \( ( \frac{m+c+1}{2}, \dots, \frac{m+c+1}{2}, \frac{m+c+3}{2}, R) \) implies that \( R\) is red.
Similarly, when \( m+c \equiv 3 \pmod{4} \) then \(R = \frac{m^2 + cm}{4} + \frac{c}{2}\) and the solution \( ( \frac{m+c+1}{2}, \dots, \frac{m+c+1}{2}, R) \) implies that \(R\) is red.

Since \( s \geq \frac{R-c}{m-1} \) we know that \( s \geq \lceil\frac{R-c}{m-1}\rceil \) and \( s > \lfloor\frac{R-c}{m-1}\rfloor \), thus \( \lceil\frac{R-c}{m-1}\rceil \) and \(  \lfloor\frac{R-c}{m-1}\rfloor \) are both red.
Then using the same argument as before 
\[ 
\Big( \underbrace{ \Big\lfloor\frac{R-c}{m-1}\Big\rfloor, \dots, \Big\lfloor\frac{R-c}{m-1}\Big\rfloor}_\text{\( c+\lceil\frac{R-c}{m-1}\rceil(m-1) - R\)}, \underbrace{ \Big\lceil\frac{R-c}{m-1}\Big\rceil, \dots,\Big\lceil\frac{R-c}{m-1}\Big\rceil }_\text{\( m - 1 - c - \lceil\frac{R-c}{m-1}\rceil(m-1) + R \)}, R, R \Big) 
\]
 creates a red solution and a contradiction.
\end{proof}

\subsubsection*{Case IV.B: 1 Red and 2 Blue}

Similar to Case III.B, we will look at the two possible colorings for 3, which gives us two sub cases.

\subparagraph{Case IV.B.1: 3 Red}

\begin{proof} This proof is identical to Case III.B.1 on page \pageref{III.2.a}.

\subparagraph{Case IV.B.2: 3 Blue}

The solution \( (2, \dots, 2, c+2m-2, c+2m-2 ) \) implies that \(c+2m-2\) is red, and \(c+3m-3\) is red since \( (3, \dots, 3, c+3m-3, c+3m-3 )\) is a solution.  
Then \( (1, \dots, 1, c+3m-3, c+2m-2 ) \) creates a red solution and a contradiction.

We must also show that \(c+2m-2\) and \(c+3m-3\) are less than \( R = \frac{m^2 + mc}{4} + \frac{c}{2} \).  
Since \(c+2m-2 \leq c+3m-3\) for \(m \geq 1\), we only need to verify the inequality involving \(c+3m-3\).

\bigskip
Note that:
\begin{table}[h!]
\begin{center}
\begin{tabular}{r  r c l r}
		& \( 3m + c - 3 \)	& \( \leq \)	& \( \displaystyle\frac{m^2 + mc}{4} + \frac{c}{2} \)	& \\
\( \iff \)	&0	   			& \( \leq \)	& \( m^2 + mc - 12m - 2c + 12 \)					& \\
\( \iff	\) 	&0	   			& \( \leq \)	& \( c(m - 2) + (m^2 - 12m + 12) ~. \)				&\ (4)
\end{tabular}
\end{center}
\end{table}

Since \(c(m-2) \geq 2\) it suffices to have \((m^2 - 12m + 12) \geq -2\), but this is equivalent to:
\begin{align*}
	22 	  &\leq (m-6)^2 \\
	11	  &\leq m~.
\end{align*}

Thus \( m < 11 \) is not covered by above we must prove each case separately.  Using (4) we can determine for which \(c\) values our above argument holds and which we have to prove.  

For each case assume 1 is red, and 2 and 3 are blue.

\newpage

\subparagraph{\underline{$m=9$:}} \indent \( 0 \leq 7c - 15 \) implies that \( c \geq 2 \).  Thus we need only consider \(c=2\)

\begin{table}[!h]
\begin{tabular}{lll}
\hspace{2.7cm} &\(R(9,2,2)=28\) & \( (2, 2, 2, 2, 2, 3, 3, 3, 3,12) \Rightarrow \) 12 is red                \\
&            & \( (2, 3, 3, 3, 3, 3, 3, 3, 3, 14) \Rightarrow \) 14 is red             \\
&            & \( (1, 1, 1, 1, 1, 1, 1, 1, 14, 12) \) creates a red solution.
\end{tabular}
\end{table}

\subparagraph{\underline{$m=7$:}} \indent \( 0 \leq 5c - 23 \) implies that \( c \geq 5 \).  Thus we need only consider \(c=2\) and \(c=4\).

\begin{table}[!h]
\begin{tabular}{lll}
\hspace{2.7cm} &\(R(7, 2, 2)=19\) & \( (2, 2, 2, 2, 3, 3, 3, 9) \Rightarrow \) 9 is red                \\
&            & \( (2, 2, 2, 3, 3, 3, 3, 10) \Rightarrow \) 10 is red              \\
&            & \( (1, 1, 1, 1, 1, 8, 8, 11) \) creates a red solution.		  \\
&	       \(R(7,4,2)=23\) & \( (2, 2, 2, 2, 3, 3, 3, 10 \Rightarrow \) 10 is red			\\
&            & \( (2, 2, 3, 3, 3, 3, 3, 11) \Rightarrow \) 11 is red              \\
&            & \( (1, 1, 1, 1, 1, 1, 10, 10) \) creates a red solution.
\end{tabular}
\end{table}

\subparagraph{\underline{$m=5$:}} \indent \( 0 \leq 3c - 23 \) implies that \( c \geq 8 \).  Thus we consider \(c=2, 4 \text{ and } 6\).

\begin{table}[!h]
\begin{tabular}{lll}
\hspace{2.7cm} & \(R(5,2,2)=11\) & \( (2, 2, 2, 2, 2, 6) \Rightarrow \) 6 is red           \\
                 &                 & \( (1, 1, 1, 1, 6, 6) \) creates a red solution.  \\
                 & \(R(5,4,2)=15\) & \( (2, 2, 2, 3, 3, 8) \Rightarrow \) 8 is red          \\
                 &                 & \( (1, 1, 1, 1, 8, 8) \) creates a red solution.  \\
	         & \(R(5,6,2)=18\) & \( (2, 2, 2, 2, 2, 8) \Rightarrow \) 8 is red          \\
	         &                 & \( (2, 2, 2, 3, 3, 9) \Rightarrow \) 9 is red   \\
                 &                 & \( (1, 1, 1, 1, 8, 9) \) creates a red solution.  
\end{tabular}
\end{table}

\subparagraph{\underline{$m=3$:}} \indent \( 0 \leq c - 15 \) implies that \( c \geq 15 \).  Thus we consider \(c=2, 4, 6, 8, \) \( 10, 12, \text{and } 14\).

\begin{table}[!h]
\begin{tabular}{lll}
\hspace{2.7cm} & \(R(3,2,2)=6\)  & \( (2, 2, 2, 4) \Rightarrow \) 4 is red      \\
                 &                 & \( (1, 1, 4, 4) \Rightarrow \) creates a blue solution.        \\
                 & \(R(3,4,2)=8\) & \( (2, 3, 3, 6) \Rightarrow \) 6 is red     \\
                 &                 & \( (1, 1, 6, 6) \) creates a blue solution.       \\
                 & \(R(3,6,2)=11\)& \( (2, 2, 2, 6) \Rightarrow \) 6 is red     \\
                 &                 & \( (2, 3, 3, 7) \Rightarrow \) 7 is red      \\
                 &                 & \( (1, 1, 6, 7) \) creates a red solution.         \\
   & \(R(3,8,2)=13\) & \( (2, 3, 3, 8) \Rightarrow \) 8 is blue     \\
                 &                 & \( (1, 1, 8, 9) \Rightarrow \) 9 is blue      \\
                 &                 & \( (2, 3, 9, 11),\Rightarrow \) 11 is red   \\
                 &                 & \( (2, 9, 9, 14),\Rightarrow \) 14 is red   \\
                 &                 & \( (1, 8, 11, 14) \) creates a red solution. 
                 \end{tabular}
\end{table}
\begin{table}[!h]
\begin{tabular}{lll} 
  \hspace{2.7cm}               & \(R(3,10,2)=16\) & \( (2, 2, 2, 8) \Rightarrow \) 8 is red     \\
                 &                 & \( (1, 1, 8, 10) \Rightarrow \)10 is blue    \\
                 &                 & \( (2, 2, 10, 12) \Rightarrow \) 12 is red    \\
                 &                 & \( (1, 1, 12, 12) \) creates a red solution.  \\
                  & \(R(3,12,2)=18\) & \( (2, 3, 3, 10) \Rightarrow \)10 is red     \\
                 &                 & \( (1, 1, 10, 12) \Rightarrow \)12 is blue     \\
                 &                 & \( (2, 3, 12, 14) \Rightarrow \) 14 is red    \\
                 &                 & \( (1, 1, 14, 14) \) creates a red solution.  \\
             	& \(R(3,14,2)=21\) & \( 2, 2, 2, 10) \Rightarrow \) 10 is red     \\
                 &                 & \( (1, 1, 10, 13) \Rightarrow \)13 is blue     \\
                 &                 & \( (2, 3, 13, 16) \Rightarrow \) 16 is red    \\
                 &                 & \( (1, 1, 16, 16) \) creates a red solution.  \\
\end{tabular}
\end{table}
\newpage
\end{proof}

We have proven the lower and upper bound making the discrete 2-color Rado number for \( x_1 + x_2 + \dots + x_m + c = 2x_0 \):
\[
  R(m, c, 2) = 
    \begin{cases}
	\infty 													&	\text{for \(m\) even, \(c\) odd} \\
	 \big\lceil \frac{m}{2} \big\lceil \frac{m+c}{2} \big\rceil + \frac{c}{2} \big\rceil		&	\text{otherwise~.}
    \end{cases}
 \]

%%%%%%%%%%%%
\section{ The Equation $x_1 + c = a x_0 $}%
%%%%%%%%%%%%

We now consider the equation \(x_1 + c = ax_0\).  
For ease of our proof, we will first consider the continuous Rado number for this equation.
Recall that in the continuous case we color all the positive real numbers from 1 to the Rado number.
Let \(a\) and \(c\) be real numbers.
Let \(R = R_\mathbb{R}(m,c,a)\) denote the smallest real number $ R$ such that for any 2-coloring of the real numbers in \([1,R]\) there exists a monochromatic solution to equation $\mathcal{E}$.  We will prove the following:

\begin{theorem}
Let \(a\) and \(c\) be real numbers with \(a \geq 1\) and \(c>0\).  Then
$$ \begin{array}{cc} 
R_\mathbb{R}(1,c,a) =
 \begin{cases}
	\infty 					&	\text{for \(a=1\)} \\
	\frac{c}{a-1}				&	\text{otherwise~.}
\end{cases}    \end{array} $$
\end{theorem}

\begin{lemma}
For \(a >1 \) and \(c>0\) we have \(R_{\mathbb{R}}(1,c,a) \geq \frac{c}{a-1} \).
\end{lemma}

\begin{proof}  Consider the following intervals for coloring:
\[ \Big[1, \frac{c+1}{a}\Big), 
\Big[\frac{c+1}{a}, \frac{c(a^2-1)+a-1}{a^2(a-1)}\Big),  
\]
\[
\Big[ \frac{c(a^2-1)+a-1}{a^2(a-1)}, \frac{c(a^3-1)+a-1}{a^3(a-1)}\Big), \dots, 
\]
\[
 \Big[ \frac{c(a^k-1) + a - 1}{a^k(a-1)}, \frac{c(a^{k+1}-1) + a - 1}{a^{k+1}(a-1)} \Big), \dots 
 \]
\par\noindent Note that $ \lim_{k \to \infty} \frac{c(a^{k+1}-1) + a - 1}{a^{k+1}(a-1)} = \frac{c}{a-1} $ and thus we have a coloring of $\left[ 1,R \right) $.

Color the number in the interval \( \Big[ \frac{c(a^k-1) + a - 1}{a^k(a-1)}, \frac{c(a^{k+1}-1) + a - 1}{a^{k+1}(a-1)} \Big)  \) red if \(k\) is even and blue if \(k\) is odd.  
We claim that this coloring will avoid a monochromatic solution: if \(x_1\) is in the interval  \(\Big[ \frac{c(a^k-1) + a - 1}{a^k(a-1)}, \frac{c(a^{k+1}-1) + a - 1}{a^{k+1}(a-1)} \Big) \), then

\begin{table}[h!]
\begin{center}
\begin{tabular}{r r c c c l}
 		 		& \( \displaystyle\frac{c(a^k-1) + a - 1}{a^k(a-1)} \) 								& \( \leq \) 	&  \( x_1 	\)			& \( < \)	& \(  \displaystyle\frac{c(a^{k+1}-1) + a - 1}{a^{k+1}(a-1)} \) \\
\( \Rightarrow \) 	& \( \displaystyle \frac{1}{a} \big( \frac{c(a^k-1) + a - 1}{a^k(a-1)} + c \big) \) 		& \( \leq \)	& \(  \displaystyle\frac{x_1 + c}{a} \)	& \( < \)	& \(  \displaystyle\frac{1}{a} \big( \frac{c(a^{k+1}-1) + a - 1}{a^{k+1}(a-1)} + c \big) \) \\
\( \Rightarrow \)	& \(  \displaystyle\frac{c(a^{k+1}-1) + a - 1}{a^{k+1}(a-1)} \)						& \( \leq \)	& \( x_0 \)			& \( <\) 	 & \(  \displaystyle\frac{c(a^{k+2}-1) + a - 1}{a^{k+2}(a-1)} ~. \)
\end{tabular}
\end{center}
\end{table}

Therefore \(x_0\) must be in the next interval, thus avoiding a monochromatic solution, making \(\frac{c}{a-1}\) a lower bound for \(x_1 + c = ax_0\).
\end{proof}

\begin{lemma}
For \(a >1 \) and \(c>0\), we have
\[
  R_{\mathbb{R}}(1, c, a) \leq
    \begin{cases}
	\infty 			&	\text{for \(a=1\)} \\
	\frac{c}{a-1}		&	\text{otherwise.}
    \end{cases}
 \]
\end{lemma}

\begin{proof} 
Starting with \(a=1\), consider the following intervals for coloring: 
\begin{center}
\( [1, c+1), [c+1, 2c+1), \dots, [kc+1, (k+1)c+1), \dots ~. \) 
\end{center}
Color the number in the interval \( [kc+1, (k+1)c+1) \) red if \(k\) is even and blue if \(k\) is odd.  We claim that this coloring will avoid a monochromatic solution.

Let \(x_1\) be in the interval \( [kc+1, (k+1)c+1) \).  Then
\begin{table}[h!]
\begin{center}
\begin{tabular}{r r c c c l}
		& \(  kc+1 \) 			& \( \leq \) & \( x_1 \) 	& \( < \) & 	\( (k+1)c+1 \) \\
\( \iff \)	& \( kc+1 +c \) 		& \( \leq \) & \( x_1+c \) & \( < \) &	\( (k+1)c+1 +c \) \\
\( \iff \)	& \( (k+1)c + 1 \)		& \( \leq \) & \( x_0 \)	& \( < \) & 	\( (k+2)c+1~. \)
\end{tabular}
\end{center}
\end{table}

Therefore \(x_0\) must be in the next interval, thus avoiding a monochromatic solution, and making \(R(1,c,1) = \infty \).

Now suppose \(a>1\).
Since \( ( \frac{c}{a-1}, \frac{c}{a-1} )\) is a solution to the equation there exists no way to color \( \frac{c}{a-1} \) to avoid a monochromatic solution.  
\end{proof}

We have proven the lower and upper bound making the continuous 2-color Rado number for \( x_1 + c = ax_0 \):
\[
  R_{\mathbb{R}} (1, c, a) = 
    \begin{cases}
	\infty 			&	\text{for \(a=1\)} \\
	 \frac{c}{a-1}		&	\text{otherwise~.}
    \end{cases}
 \]

\par
Note that for each coloring of the real numbers in the interval $[1, R )$ which avoids a monochromatic solution, restricting the coloring to just the integers in $[1, R )$ will also avoid a monochromatic integer solution to our equation. Thus, if $a-1$ divides $c$, then we will have $R (1, c, a) = \frac{c}{a-1}$.
\par
On the other hand, if $a-1$ does not divide $c$, then the discrete Rado number will be infinite. To establish this, we essentially need two colorings: one coloring for the integers from 1 to $\left\lfloor \frac{c}{a-1} \right\rfloor $, and then another for all of the integers greater than $\frac{c}{a-1}$. This is because $x_1$ and $x_0$ must lie on the same side of $\frac{c}{a-1}$:
$$x_1 < \frac{c}{a-1} \Rightarrow x_0 = \frac{1}{a} \left( x_1 + c \right) < \frac{1}{a} \left( \frac{c}{a-1} + c \right) = \frac{c}{a-1},$$
while 
$$x_1 > \frac{c}{a-1} \Rightarrow x_0 = \frac{1}{a} \left( x_1 + c \right) > \frac{1}{a} \left( \frac{c}{a-1} + c \right) = \frac{c}{a-1}.$$ 
Thus, the coloring from Lemma 3, restricted to the integers, avoids a monochromatic solution invoving integers between 1 and $\frac{c}{a-1}$. For the integers greater than $\frac{c}{a-1}$, let $t = \left\lceil \frac{c}{a-1} \right\rceil $ and consider the intervals 
$$\left[t , at - c - 1\right] , \left[ at-c , a^2 t - ac - c - 1 \right] , $$
$$\left[ a^2t - ac - c , a^3 t - a^2 c - a c - c - 1 \right] , \dots $$
$$\left[ a^k t - \frac{a^k - 1}{a-1}c , a^{k+1} t - \frac{a^{k+1} - 1}{a-1}c - 1 \right] , \dots . $$
As with the proof of Lemma 3, color the integers in the interval 
$$\left[ a^k t - \frac{a^k - 1}{a-1}c , a^{k+1} t - \frac{a^{k+1} - 1}{a-1}c - 1 \right] $$
red if $k$ is even and blue if $k$ is odd. This will avoid a monochromatic solution involving integers greater than $\frac{c}{a-1}$ since
\begin{table}[h!]
\begin{center}
\begin{tabular}{r r c c c l}
				& \( a^{k-1} t - \displaystyle\frac{a^{k-1} - 1}{a-1}c \)	& \( \leq \)	& \( x_0 \) 		& \( < \) 	& \( a^k t - \displaystyle\frac{a^k - 1}{a-1}c \) \\
\( \Rightarrow \)	& \( a^k t - \displaystyle\frac{a^k - a}{a-1}c - c \)		& \( \leq \)	& \( a x_0 -c \) 	& \( < \)	& \( a^{k+1} t - \displaystyle\frac{a^{k+1} - a}{a-1}c - c \) \\
\( \Rightarrow \)	& \( a^k t - \displaystyle\frac{a^k - 1}{a-1}c \)		& \( \leq \)	& \( x_1 \) 		& \( < \)	& \( a^{k+1} t - \displaystyle\frac{a^{k+1} - a}{a-1}c  ~. \)\\
\end{tabular}
\end{center}
\end{table}

Consequently, we have 
\[
  R(1, c, a) = 
    \begin{cases}
    	\frac{c}{a-1}	  			&	\text{if \( \left( a-1 \right)  \mid c\)} \\ 
	\infty 					&	\text{otherwise.}
    \end{cases}
 \]

%%%------------------------------------------------------------------------------------------
%%%                                               CONCLUSION
%%%------------------------------------------------------------------------------------------

%%%%%%%%%%%
\section{Conclusion}%
%%%%%%%%%%%

One limitation of our result is that we require \(c > 0\).  However, this restriction can almost certainly be loosened. There are many negative values of $c$ for which our result still holds. For example, the reader can check that for $m=50$ and $c=-10$, the arguments in Case II.A, Case II.B.1, and Case II.B.2 still work.  
\par
In the situation where $c = 2-m$, our formula yields $R(m,c,2) = 1$, which makes sense as replacing every variable with 1 yields an immediate (and necessarily monochromatic) solution. But when $c$ is slightly larger than $2-m$, the lower bound coloring in Lemma 1 can be extended. As a specific example, for equation $\mathcal{E}$ when $m=50$ and $c=-46$, coloring $\left\{ 1, 27 \right\}$ red and $ \left\{ 2, 3, \dots , 26 \right\}$ blue will avoid a monochromatic solution, and it turns out that the Rado number is 28. So, for this ``small'' value of $c$, a third block appears in the coloring, very similar to the coloring of Beutelspacher and Brestovansky in \cite{BB}. It would be nice to find for which values of $c$ this third block disappears.

\end{document}